\documentclass[11pt]{article}
\date{}

\usepackage[title]{appendix}
\usepackage{xcolor}
\usepackage[margin=1in]{geometry}                
\usepackage{graphicx}
\usepackage{subcaption}
\usepackage{pdflscape}
\usepackage{amssymb}
\usepackage[normalem]{ulem}
\usepackage{hyperref}
\usepackage{enumitem}
\usepackage{mathrsfs}
\usepackage{epstopdf}
\usepackage{rotating}
\usepackage{longtable} 
\usepackage{adjustbox}
\usepackage{float}

\usepackage{color}
\usepackage{bbm, dsfont}
\usepackage{pst-node}
\usepackage{tikz-cd}
\usepackage{amsfonts} 
\usepackage{geometry}
\usepackage{amsthm}
\usepackage{amsmath}
\usepackage{titlesec}
\usepackage{amssymb}
\usepackage{enumitem}
\usepackage{float}
\usepackage [english]{babel}
\usepackage [autostyle, english = american]{csquotes}
\usepackage{algorithm}
\usepackage[noend]{algpseudocode} 
\makeatletter
\def\BState{\State\hskip-\ALG@thistlm}
\makeatother
\usepackage{hyperref}
\usepackage[normalem]{ulem}
\usepackage{mathrsfs}
\usepackage[italicdiff]{physics}

\newlist{casess}{enumerate}{1}
\setlist[casess]{label=     \textbf{Case} \arabic*:}
\usepackage{mathtools}

\makeatletter
\newcommand*{\rom}[1]{\expandafter\@slowromancap\romannumeral #1@}
\makeatother

\usepackage{etoolbox}

\makeatletter
\patchcmd{\ttlh@hang}{\parindent\z@}{\parindent\z@\leavevmode}{}{}
\patchcmd{\ttlh@hang}{\noindent}{}{}{}
\makeatother

\usepackage{listings}
\usepackage{color} 
\definecolor{mygreen}{RGB}{28,172,0} 
\definecolor{mylilas}{RGB}{170,55,241}

\newlist{Assumptions}{enumerate}{1}
\setlist[Assumptions]{label=     \textbf{Assumption} \arabic*:}

\makeatletter

\newsavebox{\@brx}
\newcommand{\llangle}[1][]{\savebox{\@brx}{\(\m@th{#1\langle}\)}%
  \mathopen{\copy\@brx\kern-0.5\wd\@brx\usebox{\@brx}}}
\newcommand{\rrangle}[1][]{\savebox{\@brx}{\(\m@th{#1\rangle}\)}%
  \mathclose{\copy\@brx\kern-0.5\wd\@brx\usebox{\@brx}}}
\makeatother

\usepackage{lipsum} 
\usepackage{titlesec}
\titleformat{\subsection}[runin]
       {\normalfont\bfseries}
       {\thesubsection}
       {0.5em}
       {}
       [.]

 \newtheorem{thm}{Theorem}[section]
 \newtheorem{cor}[thm]{Corollary}
 
 \newtheorem{lem}[thm]{Lemma}
 \newtheorem{prop}[thm]{Proposition}
 \theoremstyle{definition}
 \newtheorem{defn}[thm]{Definition}
 \theoremstyle{remark}
 \newtheorem{rem}[thm]{Remark}
 
 \newtheorem{ex}[thm]{Example}
 
 \numberwithin{equation}{section}

\numberwithin{equation}{section}

\newcommand{\Q}{\mathbb{Q}}

\def\N{\mathbb{N}}

\def\Z{\mathbb{Z}}

\def\Z{\mathbb Z}

\DeclarePairedDelimiterX{\inp}[2]{\langle}{\rangle}{#1, #2}

\makeatletter
\newcommand*\bigcdot{\mathpalette\bigcdot@{.5}}
\newcommand*\bigcdot@[2]{\mathbin{\vcenter{\hbox{\scalebox{#2}{$\m@th#1\bullet$}}}}}
\makeatother

\def\<{\langle}
\def\>{\rangle}

\newcommand{\Cs}{\ensuremath{\mathrm{C}^\ast}}
\newcommand{\Csr}{\mathrm{C}^\ast_{\rm r}}

\numberwithin{equation}{section}

\usepackage[backend=biber,maxnames=10]{biblatex}
\begin{document}

\title{Selfless inclusions arising from commensurator groups of hyperbolic groups}

\author{Aaratrick Basu \& Felipe Flores
\footnote{
\textbf{2020 Mathematics Subject Classification:} Primary 22D25, Secondary 20F67, 46L05, 37B05
\newline
\textbf{Key Words:} Commensurator group, hyperbolic group, selfless \Cs-algebra, property $\mathrm{P}_{\mathrm{PHP}}$, Gromov boundary, extreme boundary.}
}

\maketitle

\begin{abstract}\setlength{\parindent}{0pt}\setlength{\parskip}{1ex}\noindent
We provide new examples of \Cs-selfless groups and inclusions. In particular, we prove that the commensurator group ${\rm Comm}(H)$ of a torsion-free hyperbolic group $H$ is \Cs-selfless. Our approach involves showing that the Gromov boundary $\partial H$ is a topologically free extreme boundary for ${\rm Comm}(H)$, ${\rm Aut}(H)$, and for other groups that contain $H$ in an almost normal way.
\end{abstract}

\section{Introduction}

Robert recently introduced the concept of \emph{selfless} \Cs-algebras \cite{Ro25}, provoking a small revolution in the field of \Cs-algebras. Since then, the study of selfless \Cs-algebras has provided many new examples of non-nuclear \Cs-algebras with very desirable structural properties, most notably strict comparison of positive elements (or strict comparison, for short), and topological stable rank one. 

In \cite{AGKEP25}, Amrutam-Gao-Kunnawalkam Elayavalli-Patchell proved that the reduced \Cs-algebras of finitely generated free groups are selfless and hence have strict comparison, solving a long-standing problem in operator algebras. The study of selflessness has also inspired applications to continuous model theory \cite{KES25}, von Neumann algebras \cite{HoMa26}, and the study of approximate representations of groups \cite{GKEMP26}.

The class of \Cs-algebras known to be selfless currently contains many interesting subclasses, namely some reduced group \Cs-algebras \cite{AGKEP25,Oz25,Vi25,Vi26} —which include the \Cs-algebras associated with acylindrically hyperbolic groups with trivial finite radical—, twisted reduced group \Cs-algebras \cite{RTV25,FKOCP26}, and all reduced free products, amalgamated free products, and graph products of \Cs-algebras that satisfy an appropriate version of the Avitzour conditions \cite{HKER25,FKOCP25,GKEPL26}.

The present work has the goal of producing new examples of \Cs-selfless groups and \Cs-selfless inclusions; that is, groups whose reduced \Cs-algebras are selfless or inclusions of groups that produce selfless inclusions in the sense of \cite{HKEPR25}. In order to do so, we rely on the fundamental observation of Ozawa, later extended to inclusions \cite{FKOCP26}, that groups admitting topologically free extreme boundaries are \Cs-selfless \cite{Oz25}. The main theorem of the present article provides new examples of topologically free extreme boundaries and, hence, new examples of \Cs-selfless groups.

Recall that the famous breakthrough of Breuillard-Kalantar-Kennedy-Ozawa \cite{BKKO17} characterized \Cs-simple groups, that is, groups whose reduced \Cs-algebra is simple, as the groups admitting a topologically free boundary action. These groups can also be characterized using the Furstenberg boundary, which is the universal boundary action associated to any group. \Cs-selfless groups are necessarily \Cs-simple, but the literature does not currently contain an example of a \Cs-simple group that is not \Cs-selfless, and some experts conjecture that the two classes are equal. If that is the case, then \Cs-selfless group inclusions should correspond to \Cs-irreducible inclusions. Note that the study of irreducible inclusions, initiated by R\o rdam \cite{Ro23}, is currently an active direction of research, especially for group algebras \cite{Ur22,BeOm23,LiSc23}.

Our main theorem confirms this conjecture in the case of \emph{commensurated inclusions} of hyperbolic groups by producing extreme boundary actions for these groups. Our idea is to leverage the existence of the Gromov boundary, which is, loosely speaking, a fairly universal extreme boundary naturally associated with every hyperbolic group. In fact, we follow the reasoning of Li-Scarparo \cite{LiSc23} -who worked with the Furstenberg boundary- to show that these groups act on the Gromov boundary of the almost normal subgroup and then provide a criterion for the action to be topologically free. 

Let us now state our main theorem. In it, we use the following notation: $C_G(S)$ is the centralizer of the subset $S$ in the group $G$; $\partial H$ denotes the Gromov boundary of the hyperbolic group $H$, and $H\leq_c G$ means that $H\leq G$ and $[H:H\cap gHg^{-1}]<\infty$ for all $g\in G$.

\begin{thm}\label{main}
    Let $G,H$ be countable groups, with $H$ hyperbolic and $H\leq_c G$. Suppose that $H\curvearrowright\partial H$ is topologically free and that $s\not\in C_G( H \cap s^{-1}Hs)$ for all $s\in G\setminus\{1\}$. Then $G\curvearrowright\partial H$ in a way that is topologically free and extends the action $H\curvearrowright\partial H$.
\end{thm}

It is known that the action $H\curvearrowright\partial H$ is topologically free precisely when $H$ has trivial finite radical, i.e has no normal finite subgroups. In such a case, the theorem above applies directly to the inclusion $H\leq {\rm Aut}(H)$. It seems that, until now, whether the action ${\rm Aut}(H)\curvearrowright\partial H$ is faithful or not was a problem with only partial understanding. Koberda \cite{Ko11} and later Clay-Zabanfahm \cite{ClZa17} approached this question using strong orderability assumptions. An immediate application of our main theorem settles this question in full generality.

\begin{cor}\label{gp-main}
    Let $H$ be a hyperbolic group with trivial finite radical (immediate if $H$ is torsion-free). Then the action ${\rm Aut}(H)\curvearrowright\partial H$ is topologically free.
\end{cor}

We now discuss the main application of Theorem \ref{main}. As we said before, Ozawa proved that groups admitting topologically free extreme boundaries are \Cs-selfless \cite{Oz25}. He did so by showing that they satisfy an intermediate property called property $\mathrm{P}_{\mathrm{PHP}}$. In \cite{FKOCP26}, that idea was extended to both group inclusions and twisted group \Cs-algebras. The following is an immediate application of the main theorem together with \cite[Theorem A]{FKOCP26}. The precise meaning of sefless inclusions can be found in Definition \ref{definitionsefl}. An inclusion of groups is \Cs-selfless if their reduced group \Cs-algebras form a selfless inclusion.

\begin{thm}\label{self-main}
    Let $G,H$ be countable groups, with $H$ hyperbolic and $H\leq_c G$. Suppose that $H\curvearrowright\partial H$ is topologically free and that $s\not\in C_G( H \cap s^{-1} H s)$ for all $s\in G\setminus\{1\}$. Then the inclusion $H\leq G$ is \Cs-selfless. Furthermore, for any $2$-cocycle $\omega\in Z^2(G,\mathbb T)$, the inclusion of reduced twisted group \Cs-algebras $\Csr(H,\omega|_H) \subseteq \Csr(G,\omega)$ is selfless.

    In particular, every intermediate \Cs-algebra in the inclusion $\Csr(H,\omega|_H) \subseteq \Csr(G,\omega)$ is simple, pure, has stable rank one, and a unique trace that is also the unique quasitracial state.
\end{thm}

The above theorem provides a clear set of conditions that one can check in concrete inclusions. The next corollary compiles some examples.

\begin{cor}\label{main-cor}
    The following inclusions satisfy the hypotheses of Theorem \ref{self-main}. In particular, they are \Cs-inclusions.
    \begin{enumerate}
        \item $H\leq {\rm Aut}(H)$, for any hyperbolic group $H$,
        \item $H\leq {\rm Comm}(H)$, for any hyperbolic group $H$,
        \item ${\rm PSL}(2,\Z)\leq{\rm PGL}(2,\Q)$,
        \item $\mathbb F_n\leq\mathbb F_n\rtimes_\phi \Z$, when $n>1$ and $\phi$ is an infinite order automorphism.
    \end{enumerate}
\end{cor}

This article contains three more sections. Section \ref{prelims} is devoted to discussing preliminaries on hyperbolic groups, extreme boundary actions, and \Cs-algebras. Section \ref{mainsec} contains our main results and their proofs. Besides the results that compose Theorem \ref{main} and its proof, we prove Theorem \ref{char}, which characterizes \Cs-selflessness as \Cs-irreducibility of the inclusion. Then we discuss some examples. Finally, Section \ref{comm} contains an explicit description of the abstract commensurator group, as it is our main example.

\section{Preliminaries}\label{prelims}

\subsection{Hyperbolic groups}

We summarize here some generalities on hyperbolic groups and their boundaries. For more exposition, see, for example, \cite[Section V.D]{dlH00} and the references therein.
\begin{defn}
   A finitely generated group \( H \) is \textit{(Gromov) hyperbolic}, if it admits a finite generating set \( S \) such that the word metric with respect to \( S \) makes \( H \) a (Gromov) hyperbolic metric space. Equivalently, \( H \) acts on some hyperbolic metric space properly and co-boundedly. We call \( H \) \textit{non-elementary} if it is infinite and not virtually cyclic. Unless stated otherwise, all hyperbolic groups in this paper are assumed to be non-elementary.
   \smallskip{}

   The \textit{Gromov boundary} \( \partial H \) of a hyperbolic group \( H \) is the ideal boundary of the hyperbolic metric space \( H \), i.e, the set of equivalence classes of geodesic rays under the identification of rays that stay a bounded (Hausdorff) distance apart.
\end{defn}

It is well-known that, for any hyperbolic group $H$, the boundary $\partial H$ can be topologized to make it a compact metrizable Hausdorff space, and so that \( H \cup \partial H \) is a compactification of the metric space \( H \) (cf \cite[Section H.3]{BrHa99}). Moreover, any quasi-isometry \( \varphi : K \to H \) induces a homeomorphism \( \partial \varphi : \partial K \to \partial H \) \cite[Theorem H.3.9]{BrHa99}. In particular, the action of the \( H \) on itself by conjugation induces an action on \( \partial H \). In most of the literature, it is the left multiplication action of \( H \) that is used to induce the action on \( \partial H \), but note, however, that the conjugation action induces the same maps on the boundary. This is a crucial observation for our purposes, since we will need to relate the action of \( H \) to that of certain supergroups, and the action by left multiplication is not induced by homomorphisms at the group level.

The study of many properties of hyperbolic groups is intimately connected to properties of their boundaries. Indeed, \cite{Bo98} gives a characterization of hyperbolicity in terms of the action of the group on triples of points on the boundary. We shall need the following fundamental dynamical property of the action on the boundary (cf \cite{KaBe02}).

\begin{thm}[North-south dynamics]\label{NS}
   Let \( H \) be a hyperbolic group with boundary \( \partial H \), and \( g \in H \) be an element of infinite order. Then the action of \( g \) on \( \partial H \) has exactly two fixed points \( g^{\pm\infty} \), which are obtained as the limits of \( (g^n)_{n \geq 0} \) and \( (g^{-n})_{n \geq 0} \) respectively. Moreover, for any open \( U,V \subset \partial H \) such that \( g^\infty \in U, g^{-\infty} \in V \), there is \( N \) such that \( g^n (\partial H \setminus V) \subset U \) for all \( n \geq N \).
\end{thm}

It is known that the action $H\curvearrowright\partial H$ of the hyperbolic group $H$ is topologically free precisely when it has trivial finite radical. In other words, when $H$ contains no normal finite subgroup.


\subsection{Extreme boundary actions}\label{criteria}

We now introduce the notion of extreme boundary actions. This property was observed to imply \Cs-selflessness by Ozawa \cite[Section 8]{Oz25}. In \cite{FKOCP26}, these actions are used to obtain selfless inclusions in the sense of Hayes-Kunnawalkam Elayavalli-Patchell-Robert \cite{HKEPR25}.

\begin{defn}
Let $G$ be a group, and $X$ a compact Hausdorff space with $\# X>2$. An action $G\curvearrowright X$ is referred to as an \emph{extreme boundary action} if it is minimal and \emph{extremely proximal}. Recall that $G \curvearrowright X$ is extremely proximal if, for every pair of non-empty open subsets $U,V\subseteq X$, there exists $g\in G$ such that $g(X\setminus U)\subseteq V$.
\end{defn}

The following lemma is well-known and marks the starting point of our work. It shows that actions with north-south dynamics are extremely proximal. We include its proof for the convenience of the reader.

In general, by mimicking the situation of hyperbolic groups, we will say that an element $g\in G$ has \emph{north-south dynamics} for the action $G\curvearrowright X$ if it has exactly two fixed points $x,y\in X$ and for every pair of open sets $U,V\subset X$ with $x\in U,y\in V$, there exists $n\in\N$ such that $g^n(X\setminus V)\subset U$.

\begin{lem}\label{north-south}
    Let $G\curvearrowright X$ be a continuous action of the group $G$ on the infinite compact Hausdorff space $X$. Suppose that the action is minimal and that $G$ contains an element with north-south dynamics. Then $G\curvearrowright X$ is extremely proximal.
\end{lem}
\begin{proof}
    Let $g\in G$ be the element with north-south dynamics, and let $x,y\in X$ be the associated attracting and repelling points, respectively. Let $U,V\subset X$ be non-empty open sets. By minimality, we can find elements $h_x,h_y\in G$ such that $h_xU$ is an open neighborhood of $x$ and $h_yV$ is an open neighborhood of $y$. By the north-south dynamics assumption, we may find an $n\in \N$ such that $g^n(X\setminus h_yV)\subset h_xU$, which means that $h_x^{-1}g^nh_y(X\setminus V)\subset U$. \end{proof}

The following criterion for \Cs-selflessness was observed and applied to the case of finite-index subgroups in \cite{FKOCP26}. It is originally modeled after \cite[Proposition 15]{Oz25}. The definition of selfless inclusions will be given in the next subsection.

\begin{thm}[{\cite[Theorem A, Proposition 1.4]{FKOCP26}}]\label{finind}
    Let $G$ be a group that admits a topologically free action $G\curvearrowright X$. Suppose that the restriction of the action to $H\leq G$ is an extreme boundary action. Then $H\leq G$ is a \Cs-selfless inclusion. Furthermore, for any $2$-cocycle $\omega\in Z^2(G,\mathbb T)$, the inclusion of reduced twisted group \Cs-algebras $\Csr(H,\omega|_H) \subseteq \Csr(G,\omega)$ is selfless.
\end{thm}

\subsection{Selflessness and \Cs-algebras}

The \emph{reduced group \Cs-algebra} associated with a group $G$ is denoted by $\Csr(G)$. It corresponds to the closed linear span of the operators $\{\lambda_g \mid g\in G\}\subseteq \mathbb B(\ell^2(G))$, which are defined by their action on the standard orthonormal basis of $\ell^2(G)$ by
$$
\lambda_g\delta_h=\delta_{gh}, \quad \quad\text{for all }g,h\in G.
$$
A $2$-cocycle $\omega\in Z^2(G,\mathbb T)$ may be incorporated in the construction and, as a result, one gets the \emph{twisted reduced group \Cs-algebra} $\Csr(G,\omega)$. For more details, see \cite{FKOCP26}.

The vector state associated with the vector $\delta_1$ will be denoted by $\tau$. The restriction of $\tau$ to $\Csr(G)$ is a faithful tracial state that satisfies
$$
\tau(\lambda_g)=\left\{\begin{array}{ll}
1,    & \textup{if\ } g=1, \\
0,     & \textup{if\ } g\not=1. \\
\end{array}\right.
$$
Here, $\Csr(G)$ is always equipped with the trace $\tau$.

We say that $G$ is \Cs-simple if $\Csr(G)$ is simple. A hyperbolic group is \Cs-simple precisely when it has trivial finite radical. When $G$ is \Cs-simple, the algebra $\Csr(G)$ admits a unique trace \cite{BKKO17}, so we say that $\Csr(G)$ is \emph{monotracial} and, as a result, we may omit $\tau$ from the notation.

Recall that a unital inclusion $A\subset B$ of simple \Cs-algebras is called irreducible if every intermediate \Cs-subalgebra $A\subset C\subset B$ is simple. An inclusion of groups is called \Cs-irreducible (resp. \Cs-selfless) if their reduced group \Cs-algebras form an irreducible (resp. selfless) inclusion. Selflessness for \Cs-algebras and inclusions will be defined now.

\begin{defn}[see {\cite[Definition 2.1 and Theorem 3.1]{Ro25}, \cite[Definition 2.3]{HKEPR25}}]\label{definitionsefl}
Let $(A,\rho)$ be a \Cs-algebra equipped with a GNS-faithful state. Then $(A,\rho)$ is said to be \emph{selfless}
if $A\not\cong\mathbb{C}$ and there exists a \Cs-algebra $B\not\cong\mathbb C$ equipped with a GNS-faithful state $\psi$ such that the first factor embedding into the reduced free product \Cs-algebra $\iota\colon (A,\rho) \hookrightarrow (A,\rho) \star (B,\psi)$ is \emph{existential}, that is, there exists an ultrafilter $\mathcal U$ and a state-preserving embedding $\theta\colon(A,\rho)\star(B,\psi)\hookrightarrow(A^{\mathcal U},\rho^{\mathcal U})$ such that $\theta \circ \iota$ agrees with the diagonal embedding.

Moreover, if $A_0\subseteq A$ is a \Cs-subalgebra with the property that $\theta(B)\subseteq A_0^{\mathcal U}$, where $\theta$ is the embedding witnessing that $(A,\rho)$ is selfless, then one says that the inclusion $A_0 \subseteq (A,\rho)$ is selfless. In this case, every intermediate \Cs-algebra is selfless as well.
\end{defn}

\section{Main results}\label{mainsec}


Let $G$ be a group. Two subgroups $H_1,H_2\leq G$ are said to be \emph{commensurable} if $[H_1:H_1\cap H_2]<\infty$ and $[H_2:H_1\cap H_2]<\infty$. Note that this is an equivalence relation. 

An inclusion of groups, $H\leq G$, is said to be \emph{commensurated} if, for any $g\in G$,  $ H $ is commensurable with $g H  g^{-1}$. Equivalently, for any $g\in G$, $[H:H\cap gHg^{-1}]<\infty$. In this case, we will write $H\leq_cG$. In the literature, this notion is also referred to by saying that $H$ is an \emph{almost normal subgroup} of $ G$ or that $(G,H)$ is a \emph{Hecke pair}. 

We will now prove that the action of a hyperbolic group on its Gromov boundary extends naturally to a supergroup if the inclusion is commensurated. As stated in the introduction, our proof is inspired by the fact that this occurs for the Furstenberg boundary (see \cite{BKKO17,LiSc23}).

We first define the action of \( G \) we will use. For \( g \in G \), let \( j_g \) denote the inclusion of \( H \cap gHg^{-1} \) in \( H \). By definition, this intersection has finite index in \( H \) and thus is a quasi-isometry, inducing a homeomorphism 
\[
   \partial j_g :  \partial (H \cap gHg^{-1}) \to \partial H.
\] 
We define the action of \( g \in G \) on \( \partial H \) as the homeomorphism
\[
   g_* = \partial j_g \circ \partial \iota_g \circ \left(\partial j_{g^{-1}}\right)^{-1},
\]
where \( \iota_g (x) = g x g^{-1} \) is the conjugation by \( g \), which takes \( H \cap g^{-1}Hg \) isomorphically to \( H \cap gHg^{-1} \).

\begin{thm}
   Let \( H \leq_c G \) be a commensurated inclusion, where \( H \) is hyperbolic. The definition above defines an action of \( G \) on \( \partial H \) that restricts to that of \( H \) on \( \partial H \) induced by conjugation on itself.
\end{thm}
\begin{proof}
    We get \( g_* \) is the same homeomorphism as that induced by conjugation by \( g \) whenever \( g \in H \), since the relevant inclusion maps are indentities. We now check that the definition is indeed an action of \( G \), i.e, \( (g_1 g_2)_* = (g_1)_* (g_2)_* \) for all \( g_1, g_2 \in G \). The essential observation is that 
   \[
       K = H \cap g_1^{-1}Hg_1 \cap g_2^{-1}Hg_2 \cap g_2^{-1}g_1^{-1}Hg_1g_2
   \]
   is an intersection of finite index subgroups of \( H \), and thus is still of finite index, along with the fact that \( \iota_{g_1}, \iota_{g_2}, \iota_{g_1g_2} \) satisfy \( \iota_{g_1g_2} = \iota_{g_1} \iota_{g_2} \), and in particular they induce the same map \( \partial \iota \) on boundaries. We introduce \( H_i = H \cap g^{-1}_i H g_i, i = 1,2, \) and \( H_{12} = H \cap (g_1g_2)^{-1} H g_1g_2 \), so that \( K = H_1 \cap H_2 \cap H_{12} \). We then have a diagram:
\[\begin{tikzcd}
	&& \partial H_{12} && \partial \bar H_{12} && \\
	\partial H && \partial K && \partial \bar K && \partial H \\
	& \partial H_2 & \partial \bar H_2 & \partial H & \partial H_1 & \partial \bar H_1
	\arrow[from=1-3, to=1-5, "\partial \iota_{g_1 g_2}"]
	\arrow[from=1-5, to=2-7]
	\arrow[from=2-1, to=1-3]
	\arrow[from=2-1, to=2-3]
	\arrow[from=2-1, to=3-2]
	\arrow[from=2-3, to=1-3]
	\arrow[from=2-3, to=2-5, "\partial \iota"]
	\arrow[from=2-3, to=3-2]
	\arrow[from=2-5, to=1-5]
	\arrow[from=2-5, to=2-7]
	\arrow[from=2-5, to=3-6]
	\arrow[from=3-2, to=3-3, "\partial \iota_{g_2}"']
	\arrow[from=3-3, to=3-4]
	\arrow[from=3-4, to=3-5]
	\arrow[from=3-5, to=3-6, "\partial \iota_{g_1}"']
	\arrow[from=3-6, to=2-7]
\end{tikzcd}\]
Here the bar on top indicates the image group under the respective conjugations, and the unlabelled maps are the boundary maps induced by inclusions or their inverses. By functoriality of maps induced on boundaries, the way inclusions compose, and the fact that all of the maps in the diagram are homeomorphisms, we get the diagram is commutative, and thus \( (g_1g_2)_* = (g_1)_* (g_2)_* \) as was required.
\end{proof}

Given a subset $S$ of a group $G$, let $C_G(S)$ be the \emph{centralizer} of $S$ in $G$. In the next result, we adapt the arguments of \cite[Lemma 3.3]{LiSc23} and \cite[Lemma 5.3]{BKKO17}. Note that, in their settings, the set of fixed points of a homeomorphism, here denoted ${\rm Fix}(s)$, was automatically clopen \cite{BKKO17}, and this allowed for stronger consequences.

\begin{lem}\label{interior}
Let $G,H$ be groups, with $H$ hyperbolic and $H\leq_cG$, and consider the action $ G\curvearrowright\partial H$. Given $s\in G$, if $s\in C_G( H \cap s^{-1}Hs)$, then ${\rm Fix}(s)=\partial H$.  Conversely, if $ H \curvearrowright\partial H$ is topologically free and ${\rm Fix}(s)$ has a nonempty interior, then $s\in C_G( H \cap s^{-1}Hs)$.
\end{lem}
\begin{proof}
 If $s\in C_G(H\cap s^{-1}Hs)$, then the restriction of $\iota_s$ to $H \cap s^{-1}Hs$ coincides with the identity, so $s_*|_{\partial (H \cap s^{-1}Hs)}={\rm id}_{\partial(H \cap s^{-1}Hs)}$. Since $[H:H \cap s^{-1}Hs]<\infty$, the conclusion $s_*={\rm id}_{\partial H}$ follows by density.

On the other hand, suppose that $H \curvearrowright\partial H$ is topologically free and let $\emptyset\not=U\subset{\rm Fix}(s)$ be open. Consider the set 
$$
A_U:=\{t\in H \cap s^{-1}Hs:tU\cap U\neq\emptyset\}.
$$
By Lemma \ref{north-south}, the action $H\curvearrowright\partial H$ is an extreme boundary action. Since $H \cap s^{-1}Hs$ has finite index in $H$, the action of $H \cap s^{-1}Hs$ on $\partial H$ is still an extreme boundary action \cite[Proposition 1.3]{FKOCP26}. Hence, by \cite[Lemma 5.1]{BKKO17}, the set $A_U$ generates $H \cap s^{-1}Hs$ as a group. Therefore, in order to finish the proof, we only need to show that $s$ commutes with all of $A_U$. Indeed, note that for any $t\in A_U$, the homeomorphisms $(sts^{-1})_*$ and $t_*$ coincide on $U\cap t^{-1}_*U$. However, since $H \curvearrowright\partial H$ is topologically free, there exists a point in $U\cap t^{-1}_*U$ with a trivial $H$-stabilizer, and we use this point to conclude that $t=sts^{-1}$.
\end{proof}

We now have all the necessary ingredients to prove our main theorems. We proceed to do so.

\begin{proof}[Proof of Theorem \ref{main}]
    By combining the assumptions with Lemma \ref{interior}, we get that for every non-trivial element $s\not=1$ in $G$, the fixed point set ${\rm Fix}(s)\subset \partial H$ is nowhere dense. An application of the Baire category theorem yields that $\bigcap_{s\in G\setminus\{1\}}\big(\partial H\setminus{\rm Fix}(s)\big)$ is dense in $\partial H$, from which we conclude that the action is topologically free. 
\end{proof}

\begin{proof}[Proof of Corollary \ref{gp-main}]
    Take $G={\rm Aut}(H)$ and note that the assumptions imply that the center of $H$ is trivial. Furthermore, it is easy to see that an automorphism $\phi\in{\rm Aut}(H)$ commutes with all of the inner automorphisms if and only if it is the identity.
\end{proof}

\begin{proof}[Proof of Theorem \ref{self-main}]
    Since $H\curvearrowright\partial H$ is an extreme boundary action (see Lemma \ref{north-south}) and the extension to $G$ is topologically free by Theorem \ref{main}, we have all the necessary conditions to apply Theorem \ref{finind}.
\end{proof}

Now that our main theorems are proved, we will devote the rest of the article to applications and examples. The following theorem is inspired by the philosophy that \Cs-irreducible inclusions should correspond to \Cs-selfless inclusions. In fact, it confirms that this is the case for commensurated inclusions of hyperbolic groups. It was inspired by \cite[Theorem 3.5]{LiSc23}. See also \cite{BeOm23} for the case of twisted algebras.

Recall that $H\leq G$ is said to have the \emph{relative i.c.c.\ property} if the $H$-conjugacy class
$$
g^H:=\{hgh^{-1}:h\in H\}
$$
is infinite for every $g\in G\setminus\{1\}$. A group $G$ is said to be i.c.c.\ if it is i.c.c.\ relative to itself.

\begin{thm}\label{char}
Let $H,G$ be countable groups, with $H$ hyperbolic and $H\leq_cG$. The following conditions are equivalent:

\begin{enumerate}
\item $ H \leq G $ is \Cs-irreducible;
\item $ H \leq G $ is \Cs-selfless;
\item $ H $ is \Cs-simple and $ G $ is i.c.c.\ relative to $ H $;
\item $ H $ is \Cs-simple and, for any $s\in G \setminus\{1\}$, we have that $s\notin {C}_ G ( H \cap s^{-1}Hs)$;
\item $ H $ is \Cs-simple and $ G \curvearrowright\partial H$ is topologically free;
\item $ H $ is \Cs-simple and $ G \curvearrowright\partial H$ is faithful.
\end{enumerate}

\end{thm}
\begin{proof}
The equivalence between (1), (3) and (4) was established in \cite[Theorem 3.5]{LiSc23}.
(2)$\implies$(1) follows from the fact that selfless \Cs-algebras are simple \cite{Ro25}.
(4)$\implies$(5) is Theorem \ref{main}.
(5)$\implies$(6) is trivial.
(6)$\implies$(4) follows from Lemma \ref{interior}.
(4)$\implies$(2) is Corollary \ref{self-main}.
\end{proof}

\begin{rem}
Using the rest of \cite[Theorem 3.5]{LiSc23}, one could write more equivalent conditions, including some in terms of Furstenberg boundaries and amenable URS's.
\end{rem}

\begin{ex}
    Given $n\in\N$, it was shown in \cite[Corollary 3.7]{LiSc23} that the inclusion $${\rm PSL}(n,\Z)\leq{\rm PGL}(n,\Q)$$ is commensurated and \Cs-irreducible. In the case $n=2$, ${\rm PSL}(2,\Z)$ is virtually free, hence a hyperbolic group. It follows that ${\rm PSL}(2,\Z)\leq{\rm PGL}(2,\Q)$ is \Cs-selfless.
\end{ex}

\begin{ex}
    If the hyperbolic group $H$ is \Cs-simple, then it is centerless. This means that $H$ can be identified with the group of inner automorphisms of $H$ and we get the inclusion $H\leq{\rm Aut}(H)$. It is obviously commensurated and it satisfies (4) in Theorem \ref{char}. It follows that $H\leq{\rm Aut}(H)$ is \Cs-selfless. 
    
    More generally, Bédos and Omland showed that if $H$ is any \Cs-simple group then $H\leq{\rm Aut}(H)$ is \Cs-irreducible \cite[Corollary 6.6]{BeOm23}.
\end{ex}

\begin{ex}
    If we have an extension of $H$
    $$
    1\longrightarrow H\longrightarrow G\longrightarrow K\longrightarrow 1,
    $$
    then $H\leq_cG$. In the case where the extension splits, we get a homormorphism $\phi:K\to {\rm Aut}(H)$. In this case, condition (4) in Theorem \ref{char} becomes equivalent to $\phi$ being injective. As a concrete example, one may take free-by-cyclic groups: $\mathbb F_n\leq\mathbb F_n\rtimes_\phi \Z$ is \Cs-selfless as soon as $n>1$ and $\phi$ is an infinite order automorphism.
\end{ex}


\section{Abstract commensurator groups}\label{comm}


 Let $ G $ be a group and $\Omega$ be the set of isomorphisms between finite-index subgroups of $ G $. Given $\alpha, \beta\in\Omega$,  we say that $\alpha\sim\beta$ if there exists a finite-index subgroup $H\leq {\rm dom}(\alpha)\cap{\rm dom}(\beta)$ such that $\alpha|_H=\beta|_H$.  Recall that the \emph{abstract commensurator} of $ G $, denoted by ${\rm Comm}( G )$, is the group whose underlying set is $\Omega/{\sim}$, with the product given by composition (defined up to finite-index subgroup).  

Let $ H $ be a commensurated subgroup of $ G $. Given $g\in G $, let 
\begin{align*}
\iota_g\colon  H \cap g^{-1} H  g&\to  H \cap g H  g^{-1}\\
h&\mapsto ghg^{-1}
\end{align*}
and $j_ H ^G\colon G\to{\rm Comm}( H )$ be the homomorphism given by $j^G_ H (g):=[\iota_g]$. In the case where $j_G^G$ is injective, we identify $G$ with $j_G^G(G)$.

The following facts are elementary, but one can find their proofs in \cite{LiSc23}.

\begin{prop}\label{known}
    Let $G,H$ be groups such that $ H \leq_cG$. Then the following are true.
    \begin{enumerate}
        \item Let $G$ be a group. Then $j_G^G(G)\leq_c{\rm Comm}(G)$. 
        \item The kernel of $j^G_H$ coincides with $\{g\in G:|g^ H |<\infty\}$. In particular, if $G$ is an i.c.c.\ group, then $j_G^G$ is injective.
        \item If $G$ is an i.c.c.\ group, then ${\rm Comm}(G)$ is i.c.c.\ relative to $G$.
        \item If $G$ is \Cs-simple, then $G\leq{\rm Comm}(G)$ is \Cs-irreducible.
    \end{enumerate}
\end{prop}

As a consequence, we immediately get a new example of a \Cs-selfless inclusion.

\begin{cor}
    If $H$ is a hyperbolic group with trivial finite radical, then $H\leq{\rm Comm}(H)$ is \Cs-selfless.
\end{cor}


\section*{Acknowledgments}

F.F.\ gratefully acknowledges support from the Simons Foundation Dissertation Fellowship SFI-MPS-SDF-00015100. He thanks Ben Hayes, Mario Klisse, M\'iche\'al \'O Cobhthaigh and Matteo Pagliero for continued discussions on selflessness. Both authors wish to thank Professor Thomas Koberda for his input on hyperbolic groups.

\printbibliography

\bigskip
\bigskip
ADDRESS

\smallskip
\smallskip
Aaratrick Basu

Department of Mathematics, University of Virginia,

121 Kerchof Hall. 141 Cabell Dr,

Charlottesville, Virginia, United States

E-mail: ukg7ef@virginia.edu

\smallskip
\smallskip
Felipe Flores

Department of Mathematics, University of Virginia,

114 Kerchof Hall. 141 Cabell Dr,

Charlottesville, Virginia, United States

E-mail: hmy3tf@virginia.edu
\end{document}